\documentclass[reqno]{amsart}

\usepackage{etex}
\usepackage{amsmath,amssymb,amsthm,amsfonts,mathrsfs}
\usepackage[frame,cmtip,arrow,matrix,line,graph,curve]{xy}
\usepackage{graphpap,color,paralist,pstricks}
\usepackage[mathscr]{eucal}
\usepackage{mathabx}
\usepackage[pdftex,colorlinks,backref=page,citecolor=blue]{hyperref}
\usepackage{tikz}
\usepackage{epic,eepic}
\usepackage{yfonts}
\usepackage{enumerate}

\theoremstyle{definition}
\newtheorem{theorem}{Theorem}

\newtheorem{lemma}[theorem]{Lemma}

\newtheorem{corollary}[theorem]{Corollary}

\newtheorem*{theorem*}{Theorem}
\newtheorem*{conjecture*}{Conjecture}

\theoremstyle{remark}

\begin{document}

\title{Hodge-Riemann relations for Potts model partition functions}

\author{Petter Br\"and\'en and June Huh}

\address{Department of Mathematics, KTH, Royal Institute of Technology, Stockholm, Sweden.}
\email{pbranden@kth.se}

\address{Institute for Advanced Study, Princeton, NJ, USA.}
\email{junehuh@ias.edu}
\begin{abstract}
We prove that the Hessians of nonzero partial derivatives of the (homogenous) multivariate Tutte polynomial of any matroid have exactly one positive eigenvalue on the positive orthant when $0<q\leq 1$. Consequences are proofs of the strongest conjecture of Mason and negative dependence properties for $q$-state Potts model partition functions. 
\end{abstract}

\maketitle
\thispagestyle{empty}
\section{Introduction}
Several conjectures have been made regarding unimodality and log-concavity of sequences arising in matroid theory.   Only recently have some of these been solved using combinatorial Hodge theory \cite{AHK,HSW}. A conjecture that has resisted the approach of \cite{AHK} is the strongest conjecture of Mason regarding independent sets in a matroid \cite{Mason}. The purpose of this paper is to give a self-contained proof of the strongest conjecture avoiding, but inspired by, Hodge theory. We prove that the Hessian of the homogenous multivariate Tutte polynomial (or the $q$-state Potts model partition function) of a matroid has exactly one positive eigenvalue on the positive orthant when $0<q \le 1$. 
In a forthcoming paper we will take a more general approach and see that the results proved in this paper fit into a wider context\footnote{In related  forthcoming papers, Anari, Liu, Gharan and Vinzant have independently developed methods  that overlap with our work. In particular, they also prove Mason's conjecture \eqref{MasonThird}.}. 

Let $n$ be an integer larger than $1$,
and let $\mathrm{M}$ be a matroid on $[n]=\{1,\ldots,n\}$. 
Mason \cite{Mason} offered the following three conjectures of increasing strength.
Several authors studied correlations in matroid theory partly in pursuit of these conjectures \cite{SW,Wagner08,BBL,KN10,KN11}.

\begin{conjecture*}
For any $n$-element matroid $\mathrm{M}$ and any positive integer $k$,
\begin{enumerate}[(1)]\itemsep 5pt
\item\label{MasonFirst} 
$
I_k(\mathrm{M})^2 \ge I_{k-1}(\mathrm{M})I_{k+1}(\mathrm{M}),
$
\item\label{MasonSecond} 
$
I_k(\mathrm{M})^2 \ge \frac{k+1}{k} I_{k-1}(\mathrm{M})I_{k+1}(\mathrm{M}),
$
\item\label{MasonThird} 
$
I_k(\mathrm{M})^2 \ge \frac{k+1}{k}\frac{n-k+1}{n-k} I_{k-1}(\mathrm{M})I_{k+1}(\mathrm{M}),
$
\end{enumerate}
where $I_k(\mathrm{M})$ is the number of $k$-element independent sets of $\mathrm{M}$.  
\end{conjecture*}

Conjecture (\ref{MasonFirst}) 
was proved 
 in \cite{AHK}, and Conjecture (\ref{MasonSecond}) was proved in \cite{HSW}.
Note that Conjecture (\ref{MasonThird}) may be written
 \[
\frac{I_k(\mathrm{M})^2}{{n \choose k}^2} \ge \frac{I_{k+1}(\mathrm{M})}{{n \choose k+1}}\frac{I_{k-1}(\mathrm{M})}{{n \choose k-1}}, 
\]
and the equality holds when all $(k+1)$-subsets of $[n]$ are independent in $\mathrm{M}$.
Conjecture  (\ref{MasonThird}) is known to hold when 
$n$ is at most $11$ or $k$ is at most $5$  \cite{KN11}.  
We refer to \cite{Seymour, Dowling, Mahoney, Zhao, HK,HS, Lenz} for other partial results.
We prove Conjecture (\ref{MasonThird})  in Corollary \ref{MasonIII} by uncovering concavity properties of  the multivariate Tutte polynomial of $\mathrm{M}$.

\subsection*{Acknowledgements}

Petter Br\"and\'en is a Wallenberg Academy Fellow supported by the Knut and Alice Wallenberg Foundation and Vetenskapsr\aa det. 
June Huh was supported by NSF Grant DMS-1638352 and the Ellentuck Fund.

\section{The Hessian of the multivariate Tutte polynomial}

Let $\mathrm{rk}_\mathrm{M} : 2^{[n]} \rightarrow \mathbb{Z}_{\geq 0}$ be the rank function of $\mathrm{M}$.
For a nonnegative integer $k$ and a positive real parameter $q$, consider the degree $k$ homogeneous polynomial in $n$ variables
\[
\mathrm{Z}^k_{\mathrm{M}}=\mathrm{Z}^k_{\mathrm{M}}(q,w_1,\ldots,w_n)= \sum_{A} q^{-\mathrm{rk}_\mathrm{M}(A)} \prod_{i \in A}w_i, 
\]
where the sum is over all $k$-element subsets $A$ of $[n]$. 
We define the \emph{homogeneous multivariate Tutte polynomial} of $\mathrm{M}$ by
\[
\mathrm{Z}_\mathrm{M}= 
\mathrm{Z}_{\mathrm{M}}(q,w)= \sum_{k=0}^n
\mathrm{Z}_\mathrm{M}^{n-k} w_0^{k},
\]
which is a homogeneous polynomial of degree $n$ in $w=(w_0,w_1,\ldots,w_n)$. 
When $w_0=1$, the funtion $\mathrm{Z}_\mathrm{M}$ agrees with the partition function of the $q$-state Potts model, or the random cluster model \cite{Pem,Sokal,Grimmett}.
The \emph{Hessian} of $\mathrm{Z}_\mathrm{M}$  is the matrix 
$$
\mathscr{H}_{\mathrm{Z}_\mathrm{M}}(w)= \left(\frac {\partial^2 \mathrm{Z}_\mathrm{M}}{\partial w_i \partial w_j}\right)_{i,j=0}^n. 
$$
When $w \in \mathbb{R}^{n+1}_{>0}$, the largest eigenvalue of $\mathscr{H}_{\mathrm{Z}_\mathrm{M}}$ is simple and positive 
 by the Perron-Frobenius theorem.
We prove the following analogue of the Hodge-Riemann relations for $\mathrm{Z}_\mathrm{M}$.

\begin{theorem}\label{qHR}
The Hessian of  $\mathrm{Z}_\mathrm{M}$ has exactly one positive eigenvalue for all $w \in \mathbb{R}^{n+1}_{>0}$ and $0<q \le 1$.
\end{theorem}

It follows that the Hessian of $\log \mathrm{Z}_\mathrm{M}$ is negative semidefinite on $\mathbb{R}^{n+1}_{\ge 0}$, and hence
$\log \mathrm{Z}_\mathrm{M}$ is concave on $\mathbb{R}^{n+1}_{\ge 0}$ when $0 < q \le 1$ \cite[Lemma 2.7]{AGV}.
We deduce Theorem \ref{qHR} from the following more precise statement.
Let  $c=(c_0,c_1,\ldots,c_n)$ be a sequence of $n+1$ positive real numbers.
We say that $c$ is  \emph{strictly log-concave} if 
\[
c_m^2 >c_{m-1}c_{m+1}\ \ \text{for} \ \ 0 < m <n.
\]
For any strictly log-concave sequence $c$ as above,  set
\[
\mathrm{Z}_{\mathrm{M},c}=
\sum_{k=0}^n
c_{n-k} \hspace{0.5mm}\mathrm{Z}_\mathrm{M}^{n-k} w_0^{k}.
\]
For $\alpha \in \mathbb{Z}_{\ge 0}^{n+1}$, we write
$\partial_i = \frac {\partial} {\partial w_i}$  and $\partial^\alpha = \partial_0^{\alpha_0} \partial_1^{\alpha_1}\cdots \partial_n^{\alpha_n}$.

\begin{theorem}\label{cqHR}
If $\partial^\alpha \mathrm{Z}_{\mathrm{M},c}$ is not identically zero, then
\begin{enumerate}[(i)]\itemsep 5pt
\item\label{FirstPart} the Hessian of  $\partial^\alpha \mathrm{Z}_{\mathrm{M},c}$ is nonsingular  for all $w \in \mathbb{R}^{n+1}_{>0}$ and  $0<q \le 1$, and
\item\label{SecondPart} the Hessian of  $\partial^\alpha \mathrm{Z}_{\mathrm{M},c}$ has exactly one positive eigenvalue  for all $w \in \mathbb{R}^{n+1}_{>0}$ and  $0<q \le 1$.
\end{enumerate}
\end{theorem}

Theorem \ref{qHR} can be deduced from Theorem \ref{cqHR} for $\alpha=0$ by approximating the constant sequence $1$ by strictly log-concave sequences. 
Theorem \ref{cqHR} will be proved by induction on the degree of $\partial^\alpha \mathrm{Z}_{\mathrm{M},c}$. 
For undefined matroid terminologies, see \cite{Oxley}.

\begin{lemma}\label{fundLor}
Let $A=(a_{ij})_{i,j=1}^n$ be a symmetric matrix with at least one positive eigenvalue. The following statements are equivalent.
\begin{enumerate}[(1)]\itemsep 5pt
\item $A$ has exactly one positive eigenvalue. 
\item For any $u, v \in \mathbb{R}^{n}$ with $u^TAu > 0$,  $(u^TAv)^2 \geq (u^TAu)(v^TAv).$ 
\item There is a vector $u \in \mathbb{R}^{n}$ with $u^TAu >0$, such that $(u^TAv)^2 \geq (u^TAu)(v^TAv)$ for all $v  \in \mathbb{R}^{n}$.
\end{enumerate}
\end{lemma}
\begin{proof}
Since $A$ has a positive eigenvalue, (2) implies (3). 

If (3) holds, then $A$ is negative semidefinite on the hyperplane $\{v \in \mathbb{R}^n \mid u^TAv=0\}$. Since $A$ has a positive eigenvalue, Cauchy's interlacing theorem implies (1). 

Assume (1), $u^TAu > 0$, and that $u$ and $v$ are linearly independent. Let $Q(w)=w^TAw$. The discriminant $\Delta$ of the polynomial $t \mapsto Q(tu+v)$ is $(u^TAv)^2 - (u^TAu)(v^TAv)$. If $\Delta<0$, then $Q$ is positive on the plane spanned by $u$ and $v$. This contradicts the fact that $A$ has exactly one positive eigenvalue, by Cauchy's interlacing theorem. Hence $\Delta \geq 0$, and (2) follows.
\end{proof}


\begin{lemma}\label{deg2}
Theorem \ref{cqHR} holds when the degree of $\partial^\alpha \mathrm{Z}_{\mathrm{M},c}$ is two. 
\end{lemma}

\begin{proof}
It is enough to consider the case $\partial^\alpha = \partial_0^{n-2-k}\prod_{i\in S}\partial_i$, where 
  $S$ is a $k$-element subset of $E=[n]$.  Note that $\partial_i\mathrm{Z}^\ell_{\mathrm{M}}= q^{-r(\{i\})}\mathrm{Z}_{\mathrm{M}/i}^{\ell-1}$, where $\mathrm{M}/i$ is the contraction of $\mathrm{M}$ by $i$.  We need to prove that the Hessian of the quadratic form
\[
Q=\frac {q^{\sum_{i\in S} r(\{i\})}}  {(n-k-2)!}  \partial^\alpha \mathrm{Z}_{\mathrm{M},c}=
c_k \binom {n-k} 2 w_0^2+(n-k-1) c_{k+1}\mathrm{Z}_{\mathrm{M}/S}^1(w)w_0+ c_{k+2}\mathrm{Z}_{\mathrm{M}/S}^2(w)
\]
 is nonsingular and has exactly one positive eigenvalue.  By contraction, we may assume that $S=\emptyset$ and $k=0$.  
 Write $Q(w)=w^TAw$, where $2A=\mathscr{H}_Q$. We prove that the inequality in the third statement of Lemma \ref{fundLor} is satisfied with strict inequality whenever $u=(1,0,\ldots,0)^T$, and $v\in \mathbb{R}^{n+1}$ is not a multiple of $u$. From this follows that $A$  is nonsingular and has exactly one positive eigenvalue. In other words, we will prove,  
\begin{equation}\label{discr}
 \mathrm{Z}_{\mathrm{M}}^1(w)^2 > 2 t\frac n {n-1}  \mathrm{Z}_{\mathrm{M}}^2(w) \quad  \mbox{ for all } w \in \mathbb{R}^n \setminus \{0\}, \mbox{ where } t=\frac{c_0c_{2}}{c_{1}^2}. \tag{a}
\end{equation}
Let $E_{0}$ be the set of loops in $E$, and let $E_1, E_2,\ldots, E_\ell$ be the parallel classes of $\mathrm{M}$. By the change of variables $w_j \to q w_j$ for all non-loops $j$, we get 
$\mathrm{Z}_{\mathrm{M}}^1 = e_1(E)$ and 
\begin{equation}\label{form}
\mathrm{Z}_{\mathrm{M}}^2 = e_2(E)-(1-q)(e_2(E_1)+\cdots+e_2(E_\ell)), \tag{b}
\end{equation}
where $e_k(U)$ denotes the degree $k$ elementary symmetric polynomial in the variables indexed by $U \subseteq E$. 

We prove \eqref{discr} for $t=1$ with $>$ replaced by $\geq$. Moreover, we prove that if $ \mathrm{Z}_{\mathrm{M}}^1(w)=0$ for $w\neq 0$, then $\mathrm{Z}_{\mathrm{M}}^2(w)<0$. 
The inequality \eqref{discr} for $t=\frac{c_0c_2}{c_1^2}$ then follows since $0<\frac{c_0c_2}{c_1^2}<1$. 
Note that for $q=1$ the desired inequality is an instance of the Cauchy-Schwarz inequality: 
\begin{equation}\label{CS}
(w_1+\cdots+w_n)^2 \leq n \left(w_1^2+\cdots+w_n^2\right), \quad w \in \mathbb{R}^n. \tag{c}
\end{equation}
By \eqref{form}, the inequality therefore reduces to the case when  $e_2(E_1)+\cdots+e_2(E_\ell) <0$. By monotonicity in $q$ it suffices to consider the case $q=0$.  Then the inequality reduces to 
$$
e_1(E)^2 \leq n \sum_{i=1}^\ell e_1(E_i)^2+n\!\!\sum_{j \in E_0}w_j^2,
$$
 which follows from \eqref{CS}. Suppose $\mathrm{Z}_{\mathrm{M}}^1(w)=0$ for $w\neq 0$. It remains to prove $\mathrm{Z}_{\mathrm{M}}^2(w)<0$. Since $e_1(E)=0$ and $w\neq 0$, it follows from the identity 
 $
 e_1(E)^2=2e_2(E)+ \sum_{i=1}^nw_i^2
 $
 that $e_2(E)<0$. Again the proof reduces to the case when $e_2(E_1)+\cdots+e_2(E_\ell)<0$, by \eqref{form}. We have already proved that $\mathrm{Z}_{\mathrm{M}}^2(w) \leq 0$ when $q=0$. But then $\mathrm{Z}_{\mathrm{M}}^2(w) < 0$ when $0<q\leq 1$, by \eqref{form}. This completes the  proof of the lemma. 
 \end{proof}

We prepare the proof of Theorem \ref{cqHR}  with a lemma.

\begin{lemma}\label{LefschetzInduction}
Let $F$ be a degree $d$ homogeneous polynomial in $\mathbb{R}[w_0,w_1,\ldots, w_n]$. If $w \in \mathbb{R}^{n+1}_{>0}$ and $\mathscr{H}_{\partial_i F}(w)$ has exactly one positive eigenvalue for each $i=0,1,\ldots,n$, then
\[
\ker \mathscr{H}_{F}(w)=\bigcap_{i=0}^{n} \ker \mathscr{H}_{\partial_i F}(w).
\]
\end{lemma}

\begin{proof}
We fix $w \in \mathbb{R}^{n+1}_{>0}$ and write $ \mathscr{H}_{F}$ for $\mathscr{H}_{F}(w)$.
We may suppose $d \ge 3$.
By Euler's formula for homogeneous functions, 
\[
(d-2) \hspace{0.5mm}\mathscr{H}_{F}=\sum_{i=0}^{n} w_i\hspace{0.5mm}\mathscr{H}_{\partial_i F},
\]
and hence the kernel of $\mathscr{H}_F$ contains the intersection of the kernels of $\mathscr{H}_{\partial_i F}$.

For the other inclusion, 
let $z$ be a vector in the kernel of $\mathscr{H}_F$. 
By Euler's formula again, 
 \[
(d-2) \hspace{0.5mm}e_i^T \mathscr{H}_F
= w^T\mathscr{H}_{\partial_i F},
 \]
 where $e_i$ is the $i$-th standard basis vector in $\mathbb{R}^{n+1}$, and hence
$w^T\mathscr{H}_{\partial_i F} z=0$.
We have $w^T  \mathscr{H}_{\partial_i F}  w >0$
because $w \in \mathbb{R}^{n+1}_{>0}$ and $\partial_i F$ has nonnegative coefficients.
It follows that $ \mathscr{H}_{\partial_i F} $ is negative semidefinite on the kernel of $w^T \mathscr{H}_{\partial_i F} $, by e.g. Lemma \ref{fundLor}.
In particular, 
\[
\text{$z^T  \mathscr{H}_{\partial_i F}  z \le 0$, with equality if and only if $ \mathscr{H}_{\partial_i F} z=0$.}
\]
To conclude, we write zero as the positive linear combination
\[
0=(d-2)\Big(z^T  \mathscr{H}_{F}  z\Big) = \sum_{i =0}^{n} y_i \Big(z^T \mathscr{H}_{\partial_i F} z\Big).
\]
Since every summand in the right-hand side is non-positive by the previous analysis, we must have $z^T\mathscr{H}_{\partial_i F} z=0$ for every $i$,
and hence $\mathscr{H}_{\partial_i F} z=0$ for each $i$. 
\end{proof}

\begin{proof}[Proof of Theorem \ref{cqHR}]
The proof is by induction on the degree $m$ of $F=\partial^\alpha \mathrm{Z}_{\mathrm{M},c}$. The case when $m=2$ is Lemma \ref{deg2}. By relabeling the variables we may assume that $w_0,w_1,\ldots, w_n$ are the active variables in $F$. Suppose the theorem is true when the degree of  $F$ is at most $m$, where $m \geq 2$.  

Suppose $F$ has degree $m+1$. We first prove (\ref{FirstPart}). By induction, the Hessian of any derivative of $F$ is non-singular and has exactly one positive eigenvalue. Hence (\ref{FirstPart}) for $F$ follows from Lemma~\ref{LefschetzInduction}. 

When $q=1$,  $F$ has the form 
$$
F= (\ell-1)! c_{\ell-1} e_{m+1}([n])+ \ell! c_{\ell} e_m([n])w_0 + \frac 1 2 (\ell+1)! c_{\ell +1}e_{m-1}([n])w_0^2+ \cdots. 
$$
 If we choose $c$ so that $c_i=0$ unless $i \in \{\ell-1, \ell\}$, $c_{\ell-1} =1/(\ell-1)!$ and $c_{\ell}=1/\ell!$, then $F$ is equal to the degree $m+1$ elementary symmetric polynomial in $w_0,w_1,\ldots, w_n$. The Hessian of $F$ evaluated at the all ones vector is equal to a constant multiple of 
the matrix $\mathrm{J}_{n+1}$, which has all diagonal entries equal to zero and all off-diagonal entries equal to $1$. Clearly $\mathrm{J}_{n+1}$ is nonsingular and has exactly one positive eigenvalue. We may approximate $c$ with a strictly log-concave positive sequence. 
This implies
 that that there is a strictly log-concave sequence $c$ for which the Hessian of $F$   is nonsingular and has exactly one positive eigenvalue when 
$w=(1,\ldots, 1)^T$ and $q=1$.
Since  (\ref{FirstPart}) holds for all $0<q\leq 1$ and $w \in \mathbb{R}_{>0}^{n+1}$, and (\ref{SecondPart}) holds for at least one choice of the parameters, by continuity of the eigenvalues,
 (\ref{SecondPart}) holds for all $0<q\leq 1$ and $w \in \mathbb{R}_{>0}^{n+1}$.
 \end{proof}

Theorems \ref{qHR} and \ref{cqHR} suggest that there is an algebraic structure satisfying the Poincar\'e duality and  the hard Lefschetz theorem whose degree $1$ Hodge-Riemann form is given by the Hessian of $\mathrm{Z}_\mathrm{M}$.
We refer to \cite{Huh} for a discussion of the one positive eigenvalue condition and the Hodge-Riemann relations.

\section{Consequences}

We collect some corollaries of Theorem \ref{cqHR}. It has been conjectured that the $q$-state Potts model should exhibit negative dependence properties when $0<q\leq 1$, see \cite{Pem,Sokal,Grimmett,Wagner08}. However, no substantial results on negative dependence have been proved so far. By the next theorem we see that $q$-state Potts models are \emph{ultra log-concave} for $0<q\leq 1$. 

\begin{corollary}\label{ULC}
For any $0<m<n$ and any $0<q \le 1$, we have
\[
\frac{\mathrm{Z}^m_\mathrm{M}(q,w)^2}{{n \choose m}^2} \ge \frac{\mathrm{Z}_\mathrm{M}^{m+1}(q,w)}{{n \choose m+1}}\frac{\mathrm{Z}_\mathrm{M}^{m-1}(q,w)}{{n \choose m-1}}, \quad \text{for all $w \in \mathbb{R}_{\ge 0}^n$.}
\]
\end{corollary}

\begin{proof}
Let $\mathscr{H}$ denote the Hessian of $\partial_0^{n-m-1}\mathrm{Z}_\mathrm{M}$ at $w \in \mathbb{R}_{>0}^{n+1}$. Then $(w^T\mathscr{H}e_0)^2 \geq (w^T\mathscr{H}w)(e_0^T\mathscr{H}e_0)$, where $e_0=(1,0,0,\ldots)^T$, by Theorem \ref{cqHR} and the second statement of Lemma \ref{fundLor}. By Euler's formula for homogeneous functions, 
$$
w^T\mathscr{H}e_0 = m\partial_0^{n-m} \mathrm{Z}_\mathrm{M}(w), w^T\mathscr{H}w=(m+1)m\partial_0^{n-m-1}\mathrm{Z}_\mathrm{M}(w), \mbox{ and } e_0^T\mathscr{H}e_0= \partial_0^{n-m+1}\mathrm{Z}_\mathrm{M}(w).
$$
The proof follows by continuity, letting $w_0=0$. 
\end{proof}

Let $\mathscr{I}^m_\mathrm{M}$ be the collection of independent sets of $\mathrm{M}$ of size $m$.
The \emph{$m$-th generating function} of $\mathrm{M}$ is the homogeneous polynomial in $n$ variables
\[
f^{m}_{\mathrm{M}}(w)=\sum_{I \in \mathscr{I}^m_\mathrm{M}} \prod_{i \in I} w_i,  \qquad w=(w_1,\ldots,w_n).
\]
Note that $f^m_\mathrm{M}(1,\ldots,1)$ is the number of independent sets of $\mathrm{M}$ of size $m$.

\begin{corollary}\label{MasonIII}
For every $0<m<n$ and every $w \in \mathbb{R}_{\ge 0}^n$, we have
\[
\frac{f^m_\mathrm{M}(w)^2}{{n \choose m}^2} \ge \frac{f_\mathrm{M}^{m+1}(w)}{{n \choose m+1}}\frac{f_\mathrm{M}^{m-1}(w)}{{n \choose m-1}}. 
\]
\end{corollary}

\begin{proof}
The proof is immediate from Corollary \ref{ULC} and the identity
$
f^m_\mathrm{M}(w)=\lim_{q \to 0} \mathrm{Z}^m_\mathrm{M}(q,qw).
\qedhere
$
\end{proof}
Let $\ell$ be the number of rank one flats of $\mathrm{M}$. 
The simplification $\underline{\mathrm{M}}$ of $\mathrm{M}$ is a matroid on $[\ell]$ whose lattice of flats is isomorphic to that of $\mathrm{M}$ \cite[Section 1.7]{Oxley}.
Applying Corollary \ref{MasonIII} to the simplification $\underline{\mathrm{M}}$, we get the stronger inequality
\[
\frac{f^m_\mathrm{M}(w)^2}{f_\mathrm{M}^{m+1}(w)f_\mathrm{M}^{m-1}(w)} 
 \ge \frac{{\ell \choose m}^2}{{\ell \choose m+1}{\ell \choose m-1}}
\ge \frac{{n \choose m}^2}{{n \choose m+1}{n \choose m-1}}
\ \ \text{for all $w \in \mathbb{R}^n_{\ge 0}$,}
\]



\begin{thebibliography}{Oxl11}


\bibitem[AHK18]{AHK} Karim Adiprasito, June Huh, and Eric Katz,
			\emph{Hodge theory for combinatorial geometries}.
			Ann. of Math. (2) {\bf 188} (2018), 381--452.



\bibitem[AOV]{AGV} Nima Anari, Shayan Oveis Gharan, Cynthia Vinzant,
			\emph{Log-concave polynomials, entropy, and a deterministic approximation algorithm for counting bases of matroids}.
			\texttt{arXiv:1807.00929}.


\bibitem[BBL09]{BBL} Julius Borcea, Petter Br\"and\'en, and Thomas M. Liggett, 
			\emph{Negative dependence and the geometry of polynomials}. 
			J. Amer. Math. Soc. {\bf 22} (2009), no. 2, 521--567. 


\bibitem[Dow80]{Dowling} Thomas Dowling,
			\emph{On the independent set numbers of a finite matroid}. 
			Combinatorics 79 (Proc. Colloq., Univ. Montr\'al, Montreal, Que., 1979), Part I. 
			Ann. Discrete Math. {\bf 8} (1980), 21--28. 
			

\bibitem[Gri06]{Grimmett} Geoffrey Grimmett, \emph{The random-cluster model.} Springer-Verlag, Berlin, 2006. 


\bibitem[HS89]{HS} Yahya Ould Hamidoune and Isabelle Sala\"un,
			\emph{On the independence numbers of a matroid}.
			J. Combin. Theory Ser. B {\bf 47} (1989), no. 2, 146--152. 
			
\bibitem[HK12]{HK} June Huh and Eric Katz,
			 \emph{Log-concavity of characteristic polynomials and the Bergman fan of matroids}. 
			 Math. Ann. {\bf 354} (2012), 1103--1116.
			 
\bibitem[HSW18]{HSW} June Huh, Benjamin Schr\"oter, and Botong Wang,
			\emph{Correlation bounds for fields and matroids}.
			\texttt{arXiv:1806.02675}.

\bibitem[Huh18]{Huh} June Huh,
			\emph{Combinatorial applications of the Hodge-Riemann relations}.
			Proceedings of the International Congress of Mathematicians 3 (2018), 3079--3098. 
			
\bibitem[KN10]{KN10} Jeff Kahn and Michael Neiman,
			\emph{Negative correlation and log-concavity}.
			Random Structures Algorithms {\bf 37} (2010), no. 3, 367--388. 

\bibitem[KN11]{KN11} Jeff Kahn and Michael Neiman,
			\emph{A strong log-concavity property for measures on Boolean algebras}.
			J. Combin. Theory Ser. A {\bf 118} (2011), no. 6, 1749--1760. 


\bibitem[Len13]{Lenz}  Matthias Lenz,
			\emph{The f-vector of a representable-matroid complex is log-concave}. 
			Adv. in Appl. Math. {\bf 51} (2013), no. 5, 543--545.
			
\bibitem[Mah85]{Mahoney} Carolyn Mahoney,
			\emph{On the unimodality of the independent set numbers of a class of matroids}.
			J. Combin. Theory Ser. B {\bf 39} (1985), no. 1, 77--85. 


						
\bibitem[Mas72]{Mason} John Mason, 
			\emph{Matroids: unimodal conjectures and Motzkin's theorem}. 
			Combinatorics (Proc. Conf. Combinatorial Math., Math. Inst., Oxford, 1972), 207--220, Inst. Math. Appl., Southend-on-Sea, 1972.
						
\bibitem[Oxl11]{Oxley} James Oxley, 
			\emph{Matroid theory}. 
			Second edition. Oxford Graduate Texts in Mathematics {\bf 21}. Oxford University Press, Oxford, 2011. 

\bibitem[Pem00]{Pem} Robin Pemantle, 
			\emph{Towards a theory of negative dependence}.	
			J. Math. Phys., {\bf 41}, No. 3, 1371--1390.	
	

\bibitem[Sey75]{Seymour} Paul Seymour,
			\emph{Matroids, hypergraphs, and the max-flow min-cut theorem}.
			Thesis, University of Oxford, 1975.
						
\bibitem[SW75]{SW} Paul Seymour and Dominic Welsh,
			\emph{Combinatorial applications of an inequality from statistical mechanics}.
			Math. Proc. Cambridge Philos. Soc. {\bf 77} (1975), 485--495. 
			
\bibitem[Sok05]{Sokal} Alan Sokal,  
			\emph{The multivariate Tutte polynomial (alias Potts model) for graphs and matroids}. 
			Surveys in combinatorics 2005, 173--226, London Math. Soc. Lecture Note Ser., {\bf 327}, Cambridge Univ. Press, Cambridge, 2005. 	
			
\bibitem[Wag08]{Wagner08} David Wagner,	
			\emph{Negatively correlated random variables and Mason's conjecture for independent sets in matroids}. 
			Ann. Comb. {\bf 12} (2008), no. 2, 211--239. 

\bibitem[Zha85]{Zhao} Cui Kui Zhao, 
			\emph{A conjecture on matroids}. 
			Neimenggu Daxue Xuebao {\bf 16} (1985), no. 3, 321--326. 			
								
\end{thebibliography}
\end{document}